\documentclass[12pt,a4paper]{article}
\usepackage[utf8]{inputenc}
\usepackage{amsmath}
\usepackage{amsfonts}
\usepackage{amssymb}
\usepackage{graphicx}
\usepackage{mathabx}

    \usepackage{pdfsync}

\newtheorem{theorem}{Theorem}[section]

\newenvironment{proof}[1][Proof]{\textbf{#1.} }{\ \rule{0.5em}{0.5em}}
\author{Yves Le Jan}
\title{Loop interactions and their representations in Fock space} 
\begin{document}
\maketitle

\begin{abstract}

It has been observed that on a weighted graph, an extension of Wilson's algorithm provides an independent pair $(\cal T, \cal{L})$, $\cal T$ being a spanning tree and $\cal{L}$ a Poissonian loop ensemble. This association can be interpreted in the framework of symmetric and skew symmetric Fock spaces.
Given a weighted graph, we show how  to define two types of natural interactions which correspond to local interactions between two Fock spaces. The first type of interaction involves loop ensembles and spanning trees. The second type of interaction involves loop holonomies and random connections.
\end{abstract}

\section{Framework and definitions}

Consider a system of conductances on a finite connected graph $\cal{G}=(X,E)$ without loop edges nor multiple edges. After the choice of a root $\Delta$, we denote by $E$ the set of edges not incident to $\Delta$ and by $E^o$ the set of oriented such edges. Consider the energy defined on  $X=\mathcal{X}-\Delta $ by the conductances $C$ on edges of $E$ and the killing measure $\kappa_x=C_{x,\Delta}, \;x\in X$:
$$\mathfrak{E}(f,g)=\frac{1}{2}\sum_{x,y\in X }C_{x,y}(f(x)-f(y))(\bar g(x)-\bar g(y))+\sum_{x\in X }\kappa_{x}f(x)\bar g(x)$$
We set $\lambda_x=\sum_y C_{x,y} +\kappa_x$ and denote by $M_{\lambda}$ the diagonal matrix representing the multiplication by $\lambda$.
The Green function on $X\times X$ associated with $\mathfrak{E}$ is  $G= [M_{\lambda}-C]^{-1}$.
Recall that $\mathfrak{E}(G^{x,.},G^{y,.})= G^{x,y}$ \\

 Recall that an extension of Wilson's algorithm yields an independent pair$(\cal T, \cal{L})$, $\cal T$ being a spanning tree rooted in $\Delta$ and $\cal{L}$ a Poissonian loop ensemble on $X$ with intensity given by the loop measure $\mu$ defined by the $\lambda-$symmetric continuous time Markov chain associated wth $\frak{E}$ (see \cite{stfl}).  We denote by $\mathbb P_{\cal L}$ and $\mathbb P_{\cal T}$ their distributions.\\
 The loops are obtained by dividing, at each vertex, the concatenation of the erased excursions according to a Poisson-Dirichlet distribution (in continuous time), then by forgetting the base points. We denote by $N_e(\cal L)$ the number of crossings of the edge $e$ by the loops of $\cal L$,  by $N_{e^o}(\cal L)$ the number of crossings of the oriented edge $e^o$ by the loops of $\cal L$, and by $\hat{ \mathcal{L}}^x$ the total time spent by the loops of $\cal L$ at the vertex $x$ normalized by $\lambda_x$ .  \\
Recall that for any complex function $q$, $\vert q \vert\leq 1$, defined on the set $E^o$ of oriented edges, and $\chi\geq 0$ defined on $X$,  denoting by $\circ$ the Hadamard product,
\begin{equation}\label{equ0}
E(\prod_{e^o} q_{e^o}^{N_{e^o}(\cal L)}e^{-\sum_x\chi_x\hat {\cal L}^x}) =\frac{ \det (M_{\lambda}-C)}{\det(M_{\lambda+\chi}-C\circ q)}.
\end{equation}

\section{An interaction between tree and loops}
Given a parameter $0<\beta<1$ we can define an interacting pair $(\cal T,\cal L)$ by  the joint distribution:
$$\mathbb{P}^{(\beta)}_{\cal T, \cal L}(T,dL)=\frac{1}{Z^{(\beta)}} \prod _{e\not \in T}\beta^{N_e(\cal L)}\,\mathbb P_{\cal L}(dL)\,\mathbb P_{\cal T}(T),$$
$Z^{(\beta)}$ being a normalization constant.\\
As $\beta$ tends to 0, the loops of $\cal L$ tend to be carried by $\cal T$. In particular, they tend to have trivial holonomies, i.e. to be contractible to a point. As $\beta$ tends to 1, $\cal T$ and $ \cal{L}$ tend to be independent.\\
Similarly, Given a parameter $b>0$ we can define an interacting pair $(\cal T,\cal L)$ by  the joint distribution:
$$\mathbb{P}^{(b*)}_{\cal T, \cal L}(T,dL)=\frac{1}{Z^{(b*)}} \prod _{x, \{x\Delta \}\not \in T}e^{-b\hat{\mathcal {L}}^x}\,\mathbb P_{\cal L}(dL)\,\mathbb P_{\cal T}(T),$$
$Z^{(b*)}$ being a normalization constant.

 \section{Interaction in supersymmetric Fock space}
  The independent pair $(\cal T, \cal{L})$, associating a spanning tree   $\cal T$ and a Poissonian loop ensemble $\cal{L}$ can be interpreted in the framework of symmetric and skew symmetric Fock spaces (see \cite{stfl}).
 
 The partition function and more generally expectations of various functionals of the random pair $(\cal T,L)$ can be expressed  in terms of the supersymmetric Fock space associated with $G$. 
 First note that if $\phi$ denotes the complex Bose field and $\psi$ the Fermi field, it follows from (\ref{equ0}) and from Fock space calculations (see for example \cite{fock} and \cite{stfl}) that
 for any complex function $q$, $\vert q \vert\leq 1$, defined on the set $E^o$ of oriented edges, and $\chi\geq 0$ defined on $X$, denoting the vacuum state by 1,
 \begin{equation}\label{eq2} \begin{split}
 E(\prod_{e^o} q_{e^o}^{N_{e^o}(\cal L)}e^{-\sum_x\chi_x\hat {\cal L}^x}) =\langle 1,\; \exp(\frac{1}{2}\sum_{x,y}C_{x,y} [q_{x,y}-1] \phi(x)\bar\phi(y)-\frac{1}{2}\sum_{x}\chi_x \phi(x)\bar\phi(x )) \;1\rangle.  \\
\end{split}
 \end{equation}
  Note that the same representation can be given in terms of expectation of functionals of  complex Gaussian variables. This is in fact the usual terminology in probability but we are using Bose fields to emphasize the symmetry with the Fermi field.\\
 Also, for any function $b$ and $c$ defined on edges, setting $$\vert d\psi_{\{x,y\}}\vert^2=(\psi(x)-\psi(y))(\bar\psi(x)-\bar \psi(y))$$,
  \begin{equation}\label{eq3} \begin{split}
 E&(\prod_e (b_e1_{e\not\in \mathcal{ T}}+c_e 1_{e\in \mathcal{ T}}))=\langle 1,\; \prod_e(b_e(1-\vert d\psi_e\vert^2)+c_e \vert d\psi_e\vert^2
  \;1\rangle 
  \end{split}
  \end{equation}
   Note that the same representation can be given in terms of  complex differential forms (see the introduction of \cite{fock}), or in terms of and Grassmann integration (\cite{B}).\\
  Then we have the following representation of $\mathbb{P}^{(\beta)}_{\cal T, \cal L}$:
  \begin{theorem}\label{t1}
  For any $0<\beta<1$,\\
   $Z^{(\beta)}=\langle 1,\; \prod_{\{x,y\}}[(1-\vert d\psi_{\{x,y\}}\vert^2)e^{-\frac{1}{2}C_{x,y} [\beta-1](\phi(x)\bar\phi(y)+\phi(y)\bar\phi(x))} \vert+ d\psi_{\{x,y\}}\vert^2] \;1\rangle$\\
   
   More generally, for any $b,c$  defined on edges, $\chi \geq 0$ defined on vertices and any complex function $q$, $\vert q \vert\leq 1$, defined on oriented edges, the expression $ \int \prod_{e^o} q_{e^o}^{N_{e^o}(L)}(b_e1_{e\not\in \mathcal{ T}}+c_e 1_{e\in \mathcal{ T}})e^{-\sum_x\chi_x\hat L^x})\mathbb{P}^{(\beta)}_{\cal T, \cal L}(T,dL))$ equals:\\
   
   $\frac{1}{Z^{(\beta)}}\langle 1,\;e^{-\frac{1}{2}\sum_x\chi_x \phi(x)\bar\phi(x)}
    \prod_{\{x,y\}} [b_{\{x,y\}}e^{\frac{1}{2}C_{x,y}( [ \beta q_{(x,y)}-1]\phi(x)\bar\phi(y)+[ \beta q_{(y,x)}-1]\phi(y)\bar\phi(x))}(1-\vert d\psi_{\{x,y\}}\vert^2)
   + c_{\{x,y\}}  \prod_{\{x,y\}} e^{\frac{1}{2}C_{x,y}( [ q_{(x,y)}-1]\phi(x)\bar\phi(y)+[ q_{(y,x)}-1]\phi(y)\bar\phi(x))} \vert d\psi_{\{x,y\}}\vert^2]
   \;1\rangle$
\end{theorem}
  \begin{proof}
  The first expression equals:   $\frac{1}{Z^{(\beta)}}\int\prod_e (b_e  \beta^{N_e( L)} 1_{e\not\in \mathcal{ T}}+c_e  1_{e\in \mathcal{ T}}) \prod_{e^o} q_{e^o}^{N_{e^o}(L)}e^{-\sum_x\chi_x\hat { L}^x} \\
  \,\mathbb P_{\cal L}(dL)\,\mathbb P_{\cal L}(T)$\\
  
 $ =\frac{1}{Z^{(\beta)}}\int\prod_{x, y}\;[1_{\{x,y\}\not\in T}\;b_{\{x,y\}} (\beta q_{(x,y)})^{N_{(x,y)}(L)} +1_{\{x,y\}\in T}\;c_{\{x,y\}} ( q_{(x,y)})^{N_{(x,y)}(L)}]\;e^{-\sum_x\chi_x\hat { L}_x} \\
  \mathbb P_{\cal L}(dL)\,\mathbb P_{\cal L}(T)$\\
  
 $=\frac{1}{Z^{(\beta)}}\int \langle 1,\; e^{-\frac{1}{2}\sum_x\chi_x \phi(x)\bar\phi(x)}\prod_{{\{x,y\}}\not\in T}b_{\{x,y\}}e^{\frac{1}{2}C_{x,y}( [ \beta q_{(x,y)}-1]\phi(x)\bar\phi(y)+[ \beta q_{(y,x)}-1]\phi(y)\bar\phi(x))}\\ \prod_{{\{x,y\}}\in T}c_{\{x,y\}}e^{\frac{1}{2}C_{x,y}( [ q_{(x,y)}-1]\phi(x)\bar\phi(y)+[ q_{(y,x)}-1]\phi(y)\bar\phi(x))} 
  \;1\rangle\,\mathbb P_{\cal L}(T )$\\
  
 $\frac{1}{Z^{(\beta)}}\langle 1,\;e^{-\frac{1}{2}\sum_x\chi_x \phi(x)\bar\phi(x)}
    \prod_{\{x,y\}} [b_{\{x,y\}}e^{\frac{1}{2}C_{x,y}( [ \beta q_{(x,y)}-1]\phi(x)\bar\phi(y)+[ \beta q_{(y,x)}-1]\phi(y)\bar\phi(x))}(1-\vert d\psi_{\{x,y\}}\vert^2)
   + c_{\{x,y\}}  \prod_{\{x,y\}} e^{\frac{1}{2}C_{x,y}( [ q_{(x,y)}-1]\phi(x)\bar\phi(y)+[ q_{(y,x)}-1]\phi(y)\bar\phi(x))} \vert d\psi_{\{x,y\}}\vert^2]
   \;1\rangle$
  
  \end{proof}\\
  Note that for $\beta $ close to 1, the joint distribution $\mathbb{P}^{(\beta)}_{\cal T, \cal L}$ is a perturbation of the product $\mathbb P_{\cal L}\otimes \mathbb P_{\cal T}$.  The Fock space representation allows to expand the partition function and related expressions according to powers of $1-\beta$.\\
  A similar representation can be given for $\mathbb{P}^{(b*)}_{\cal T, \cal L}$.
%
\section{Connections and  holonomy}
Consider a finite group $M$ and a $M$-connection $A$ on the graph. A $M$-connection can be defined as an equivalence class of maps $m$ from oriented edges into $M$, such as opposite orientations have inverse images. $m$ is equivalent to $m'$ if and only if there exist a map $h$ from vertices into $M$ such that $m'_{x,y}=h_x m_{x,y}h_y^{-1}$. The choice of a representative $m$ of a connection $A$ is often refered to as a choice of gauge. A connection defines a non-ramified cover of the graph. Fibers have cardinality $\vert M \vert$ and $M$ acts faithfully and transitively on them. The conductances and the killing measure can be lifted to the covering graph. We denote by $G^{(A)}$ the associated Green function and by $\mathcal{L}^{(A)}$ the associated loop ensemble.  \\
If $M= \mathbb Z/ 2  \mathbb Z$ , we note that connections are defined by percolation configurations.\\Given a connection $A$, any loop $l$ defines a conjugacy class of $M$ , denoted $H_A(l)$. It is obtained by choosing a base point in $l$, some $m$ representing $A$,  by multiplying the group elements assigned to the edges of the loop in cyclic order and by taking the conjugacy class of the product. Clearly, the holonomy depends only on the geodesic (i.e. non-backtracking) loop associated with $l$ by removing tree-like subloops. Geodesic loops represent the conjugacy classes of the fundamental group.  \\The projection of the loop ensemble $\mathcal{L}^{(A)}$ on the graph $\cal{G}$ is the set of loops of trivial holonomy in  $\mathcal{L}_{\vert M\vert}$, the union of $\vert M\vert$ independent copies of  $\mathcal{L}$ (which is a Poisson process with intensity $\vert M\vert \mu$) (see \cite{gaugeloops}). Denoting by $\iota$ the unity of $M$, image of $\mathbb P_{\mathcal{L}^{(A)}}$  is $\frac{\det(G)^{\vert M\vert}}{\det(G^{(A)})}  \prod_{l \in \mathcal{L}_{\vert M\vert}}  1_{\iota}[H_A(l)]) \mathbb P_{\mathcal{L}_{\vert M\vert}}(dL)$.
Conversely, the loop ensemble $\mathcal{L}^{(A)}$ can be constructed by taking independently and uniformly a lift of all loops of trivial holonomy in  $\mathcal{L}_{\vert M\vert}$.

The counterpart of this property in Fock space is that in a given gauge, the density of the Gaussian free field on the cover with respect to the densities of $\vert M\vert$ independent free fields $\phi_i, \; i\in M$ on $X$ is given by :\\

$$\prod_{x\neq y}\exp(\sum_{i,j} C_{x,y} (\delta_{i,m\cdot j}-\delta_{i,j})\phi_i(x)\bar\phi_j(y))$$

In particular,\\
\begin{equation}\label{equ2}
 \begin{split}
&E(\prod_{e^o} q_{e^o}^{N_{e^o}(\mathcal{ L}^{(A)})}e^{-\sum_x\chi_x\hat { \mathcal{L}}^x_{\vert M\vert}})=E(\prod_{e^o} q_{e^o}^{N_{e^o}(\mathcal{ L}_{\vert M\vert})} e^{-\sum_x\chi_x\hat{ \mathcal{L}}^x_{\vert M\vert}}\prod_{l \in \cal L} 1_{\iota}[H_A(l)])\\
&=\langle 1,\;  \prod_{x\neq y}e^{(\sum_{i} C_{x,y} (q_{x,y}-1 )\phi_i(x)\bar\phi_i(y)}e^{-\frac{1}{2}\sum_{x,i}\chi_x \phi_i(x)\bar\phi_i(x)}\prod_{x\neq y}e^{\sum_{i,j} C_{x,y}  \delta_{i,m\cdot j}-\delta_{i,j })\phi_i(x)\bar\phi_j(y)}\;1\rangle .\\
&=\langle 1,\; e^{-\frac{1}{2}\sum_{x,i}\chi_x \phi_i(x)\bar\phi_i(x)}\prod_{x\neq y}e^{\sum_{i,j} C_{x,y}  q_{x,y}(\delta_{i,m\cdot j}-\delta_{i,j })\phi_i(x)\bar\phi_j(y)}\;1\rangle . 
\end{split}
\end{equation}

%

\section{Interaction with connections}
Given a spanning tree $T$, we say that $m$ is $T$-reduced if $m_{x,y}=\iota$ for all edges $\{x,y\}$ of $T$. One easily sees that any connection has a unique $T$-reduced representative. If $\gamma$ is a symmetric probability on $m$, assigning to the edges of the tree, oriented arbitrarily, independent $\gamma$-distributed random elements of $M$ defines a random connection $\cal A$. Its distribution $\gamma^T$ does not depend on the chosen orientation. Then $\sum_T\mathbb P_{\mathcal T}(T)\gamma^T$ is a natural distribution on the space of connections and $(T,A)\rightarrow\mathbb P_{\cal T}(T)\gamma^T(A) $ a joint probability distribution on spanning tree and connections. The tree and the connection are generally not independent, unless $\gamma$ is chosen to be uniform.\\ 
Any non-negative central function $\Phi$ on $M$ defines a distribution on triples $(T,A,L)$, $L$ denoting a countable family of time continuous loops on the graph, not visiting $\Delta$:
$$\nu_{\Phi}(T,A,dL)=\frac{1}{Z_{\Phi}}\prod_{l\in L}\Phi(H_{A}(l)\gamma^{T}(A)\mathbb P_{\cal T}(T) \mathbb P_{\mathcal{ L}_{\vert M\vert}}(dL).$$

$Z_{\Phi}$ denotes a normalization constant (partition function).\\
 $\Phi$ can be decomposed using the characters of the unitary representations of $M$.\\
 
A natural choice of $\Phi$ is $1_{\iota}$. Then the partition function is $$\sum_{T,A}\frac{\det(G)^{\vert M \vert}}{\det(G^{(A)})}\mathbb P_{\cal T}(T)\gamma^T(A) .$$
 Note that if $\gamma$ is close to $\delta_{\iota}$,  the joint distribution of $ T$ and $L$ is close to $\mathbb P_{\mathcal{ L}_{\vert M\vert}}\otimes \mathbb P_{\cal T}.$  The Fock space representation can yield a perturbation expansion of the partition function and related expressions.
%



\bigskip

\noindent
NYU Shanghai. 1555 Century Blvd, Pudong New District. Shanghai. China. \\
and\\
  D\'epartement de Math\'ematique. Universit\'e Paris-Sud.  Orsay, France.\\
  
   \bigskip   
\noindent 
   yves.lejan at math.u-psud.fr\\
   and\\ 
   yl57 at nyu.edu

\end{document}